\documentclass[12pt,reqno]{amsart}
\usepackage[activeacute,english]{babel}

\usepackage[utf8]{inputenc}
\usepackage{amsmath} 
\usepackage{amsthm} 
\usepackage{amssymb} 
\usepackage{amscd} 

\usepackage{emptypage}
\usepackage{enumitem} 
\usepackage{mathtools} 
\usepackage{mathrsfs} 
\usepackage{upgreek}

\usepackage{hyperref}
\usepackage{cleveref}
\usepackage{autonum}

\usepackage{esint} 
\usepackage{graphicx,caption} 
\usepackage{float} 

\usepackage{pgf,tikz,pgfplots}
\usetikzlibrary{arrows}
\usetikzlibrary{intersections}
\usetikzlibrary{fadings}
\usepackage{wrapfig}
\usepackage{xcolor}
\definecolor{CBgreen}{HTML}{009E73}
\definecolor{CBorange}{HTML}{E69F00}

\usepackage[margin=0.9in]{geometry}
\newtheorem{theorem}{Theorem}[section]
\newtheorem{proposition}[theorem]{Proposition}
\newtheorem{lemma}[theorem]{Lemma}
\newtheorem{corollary}[theorem]{Corollary}

\theoremstyle{definition}

\theoremstyle{remark}
\newtheorem{remark}[theorem]{Remark}
\newtheorem*{remark*}{Remark}

\newcommand{\con}[1]{\mathbb{#1}}
\newcommand{\C}{\con{C}} 
\newcommand{\R}{\con{R}} 

\newcommand{\D}{\mathcal{D}}
\newcommand{\RR}{\mathcal{R}}

\newcommand{\Dom}{\mathrm{Dom}}

\newcommand{\disk}{\con{D}} 

\numberwithin{equation}{section}
\title[Eigenvalue bounds for quantum dot Dirac operators]{Eigenvalue bounds for quantum dot Dirac operators}

\author[J. Duran]{Joaquim Duran}
\address{J. Duran
\newline
Centre de Recerca Matem\`atica, Edifici C, Campus Bellaterra, 08193 Bellaterra, Spain}
\email{jduran@crm.cat}

\date{\today}

\thanks{The author is supported by the Spanish grant PID2025-168310NB-I00 funded by MICIU/AEI/10.13039/501100011033 and by FSE+, by ERDF ``A way of making Europe", and by the Catalan grant 2021-SGR-00087. This work is supported by the Spanish State Research Agency, through the Severo Ochoa and Mar\'ia de Maeztu Program for Centers and Units of Excellence in R\&D (CEX2020-001084-M), and more specifically by the grant CEX2020-001084-M-20-1. The author acknowledges CERCA Programme/Generalitat de Catalunya for institutional support.}

\begin{document}

\begin{abstract}
    We exploit the connection between quantum dot Dirac operators and~$\overline\partial$-Robin Laplacians. First, we find a graphical relation between their smallest positive eigenvalues, which allows us to deduce a recipe for translating bounds (from above and below) from one to the other. As an application, we provide new upper and lower bounds for the eigenvalues of the quantum dot Dirac operators, which depend only on geometric quantities of the underlying domain. In particular, we obtain some Faber-Krahn type inequalities for convex thin domains.
\end{abstract}

\maketitle 

\setcounter{tocdepth}{2}
\makeatletter
\def\l@subsection{\@tocline{2}{0pt}{2.5pc}{5pc}{}}
\makeatother
\tableofcontents

\section{Introduction}

Two-dimensional Dirac operators subject to infinite mass or zigzag boundary conditions are employed in the physics literature to model the behavior of electrons responsible for electrical conduction in graphene quantum dots and nano-ribbons~\cite{AkhmerovBeenakker,BerryMondragon1987,mccannfalko04,WurmRycerzetAl}. From the mathematical perspective, Dirac operators with zigzag boundary conditions were first studied in~\cite{Schmidt1995}, while a broader class of Dirac operators ---including those with infinite mass boundary conditions--- was studied in~\cite{Benguria2017Self,Benguria2017Spectral,Cassano2023}. Such a family of Dirac operators ---from now on called \emph{quantum dot Dirac operators}, and defined in~\eqref{def:Dirac_op_theta}--- was revisited in~\cite{DuranMasSanzPerela2026} with strong inspiration from the previous work~\cite{Antunes2021}. More specifically, in~\cite{DuranMasSanzPerela2026} a connection was shown between the eigenvalues of the quantum dot Dirac operators and the eigenvalues of the so-called \emph{$\overline\partial$-Robin Laplacians} ---studied in~\cite{Duran2026}, and defined in~\eqref{eq:RodzinLaplacian}.

In this work we build on such a connection to improve the results presented in~\cite[Section~2.1]{DuranMasSanzPerela2026}, obtaining an approach to get upper and lower bounds for the smallest positive eigenvalue of the quantum dot Dirac operators. As a direct application, we derive some Faber-Krahn type inequalities for these operators which cover, for instance, the class of convex thin domains; but not only, see~\Cref{rmk:examples}. Although the papers~\cite{Antunes2021,Antunes2024} provided some numerical evidence, to the best of our knowledge this represents the first instance in which the minimality of the disk is shown with respect to a wide class of domains for two-dimensional Dirac operators subject to infinite mass boundary conditions. 

The precise statements of the results are presented in~\Cref{sec:main_results}. Before that, in the following sections we set the stage and recall the operators appearing in this paper.

\textbf{Note:} during the peer-review process of this work, the recent paper~\cite{Ammari2026} appeared claiming the proof (with AI assistance tools) of the strict Faber-Krahn inequality for two-dimensional Dirac operators on $C^\infty$ simply connected domains subject to infinite mass boundary conditions. This not only covers our partial results on the minimality of the disk with respect to some classes of domains for the infinite mass boundary conditions, but also solves what was a hot open problem in spectral geometry; see~\cite[Problem~5.1]{ProblemListShapeOptimization}. However, our partial results on the Faber-Krahn inequality are also valid for quantum dot boundary conditions other than the infinite mass, and do not need neither the $C^\infty$ nor the simple connectivity of the underlying domain.

\subsection{The setting} \label{sec:setting}

Throughout the work, $\Omega\subset \R^2$ is a bounded domain with $C^2$ boundary; see~\Cref{rmk:Lipschitz} for a justification on this regularity assumption and a comment regarding domains with Lipschitz boundary. We denote the unit outward normal vector field at~$\partial\Omega$ as~$\nu=(\nu_1,\nu_2)$, and the unit tangent vector field positively oriented to it as~$\tau=(-\nu_2,\nu_1)$. Based on the identification~$\R^2\equiv \C$, we abuse notation and write
\begin{equation}
    \R^2 \ni p = (p_1,p_2) \equiv p_1+ip_2 = p \in \C, \quad \text{and, accordingly,} \quad \overline p = p_1-ip_2;
\end{equation}
it will always be clear from the context which notation we are referring to. Moreover, we use the complex notation~$\partial_z := \frac{1}{2} (\partial_1-i\partial_2)$ and~$\partial_{\bar z} := \frac{1}{2}(\partial_1+i\partial_2)$, where~$\nabla:=(\partial_1,\partial_2)$ denotes the gradient in~$\R^2$.

Usual geometric quantities of $\Omega$ appear along the work, like its area~$|\Omega|$, its perimeter~$|\partial\Omega|$, and also
\begin{equation} \label{eq:FirstDirichlet}
    \Lambda_\Omega := \text{the first (smallest) eigenvalue of the Dirichlet Laplacian in } \Omega.
\end{equation}
Further geometric quantities of $\Omega$ that only appear in specific parts of this paper are discussed in~\Cref{sec:bounds,sec:FK}.

\subsection{The quantum dot Dirac operator} \label{sec:QD}

Given a parameter $m\geq0$ ---that usually denotes the mass---, let $-i \sigma \cdot \nabla + m \sigma_3$ denote the differential expression of the free Dirac operator in~$\R^2$. Here $\sigma := (\sigma_1, \sigma_2)$, where
\begin{equation}
    \sigma_1 := 
    \begin{pmatrix}
	    0&1\\
        1&0
    \end{pmatrix}, \quad
    \sigma_2 := 
    \begin{pmatrix}
	    0&-i\\
        i&0
    \end{pmatrix}, \quad
    \sigma_3 := 
    \begin{pmatrix}
	    1&0\\
        0&-1
    \end{pmatrix}
\end{equation}
are the Pauli matrices, and we are denoting $\sigma \cdot p:=\sigma_1 p_1+\sigma_2 p_2$. Given~$\theta\in(-\frac \pi 2,\frac {\pi}{2})$, the \textbf{quantum dot Dirac operator}~$\D_{\theta,m}$ is the operator in~$L^2(\Omega,\C^2)$ defined by
\begin{equation} \label{def:Dirac_op_theta}
	\begin{split}
		\mathrm{Dom}(\D_{\theta,m}) &:= \big\{ \varphi \in H^1(\Omega,\C^2): \, \varphi =(\cos\theta\,\sigma\cdot \tau+\sin\theta\, \sigma_3) \varphi  \,\text{ in } H^{1/2}(\partial \Omega,\C^2) \big\},\\
		\D_{\theta,m}\varphi &:= 
		(-i\sigma\cdot\nabla+m\sigma_3)\varphi \quad\text{for all 
			$\varphi\in\mathrm{Dom}(\D_{\theta,m})$}.
	\end{split}
\end{equation}
The infinite mass boundary condition corresponds to $\theta=0$. The zigzag boundary conditions arise from the self-adjoint realizations of~$\D_{\theta,m}$ in the limits $\theta\to \pm\frac{\pi}{2}$. In this case their spectra are fully described in terms of the Dirichlet Laplacian (which is well understood), and for this reason we do not consider them in this work; see~\cite{Schmidt1995}.

It is known that, for all~$\theta\in(-\frac \pi 2,\frac {\pi}{2})$ and all~$m\geq0$, the operator~$\D_{\theta,m}$ is self-adjoint in~$L^2(\Omega,\C^2)$ and its spectrum consists of isolated eigenvalues of finite multiplicity accumulating only at $\pm\infty$; see~\cite[paragraph before equation (1.2)]{DuranMasSanzPerela2026} and  also~\cite{Benguria2017Self}. Moreover, the negative eigenvalues of~$\D_{\theta,m}$ are characterized by the positive eigenvalues of~$\D_{-\theta,m}$; see the unitary equivalences discussed in~\cite[Appendix~A.2]{DuranMasSanzPerela2026} and~\cite[paragraph before equation~(1.3)]{DuranMasSanzPerela2026}. Therefore, to study the spectrum of $\D_{\theta,m}$ as $\theta$ ranges in $(-\frac \pi 2,\frac {\pi}{2})$ one can reduce to the eigenvalue problem
\begin{equation} \label{eq:QD_eigen}
    \begin{cases}
        -2i \partial_z v = (\lambda-m) u & \text{in } \Omega, \\
        -2i \partial_{\bar z} u = (\lambda+m)v & \text{in } \Omega, \\
        \overline \nu v = i \vartheta(\theta) u & \text{on } \partial \Omega,
    \end{cases}
\end{equation}
assuming that~$\lambda\geq 0$. Here, $\vartheta$ is the smooth, strictly decreasing, and bijective function 
\begin{equation} \label{def:vartheta}
    \begin{split}
        \vartheta: \textstyle{ (-\frac \pi 2, \frac \pi 2) } & \to (0,+\infty)\\
        \theta & \mapsto \textstyle{ \frac{1-\sin\theta}{\cos\theta} };
    \end{split}
\end{equation}
notice that \eqref{eq:QD_eigen} is the eigenvalue problem $\D_{\theta,m} \varphi = \lambda \varphi$ associated to~\eqref{def:Dirac_op_theta} when we write an eigenfunction~$\varphi$ in components, namely $\varphi=(u,v)^\intercal$ with $u,v:\Omega\to \C$. 

Among all the nonnegative eigenvalues of~\eqref{eq:QD_eigen}, in this work we are interested in the smallest one. For all~$\theta\in(-\frac \pi 2,\frac {\pi}{2})$ and all~$m\geq0$, we denote
\begin{equation}
    \lambda_\Omega(\theta,m) := \text{the first (smallest) nonnegative eigenvalue } \lambda \text{ of } \eqref{eq:QD_eigen}.
\end{equation}

\subsection{The $\overline\partial$-Robin Laplacian} \label{sec:dbarRobin}

Given $a>0$, the \textbf{$\overline\partial$-Robin Laplacian}~$\RR_a$ is the operator in~$L^2(\Omega,\C)$ defined by
\begin{equation} \label{eq:RodzinLaplacian}
    \begin{split}
        \Dom(\RR_a) & := \big\{u\in H^1(\Omega,\C): \, \partial_{\bar z} u \in H^1(\Omega,\C), \, 2\bar \nu \partial_{\bar z}u + au = 0 \text{ in } H^{1/2}(\partial \Omega,\C) \big\}, \\
        \RR_a u & := -\Delta u \quad \text{for all } u \in \Dom(\RR_a).
    \end{split}
\end{equation}
For all~$a>0$, the operator $\RR_a$ is self-adjoint in $L^2(\Omega,\C)$, and its spectrum consists of isolated and positive eigenvalues of finite multiplicity accumulating only at $+\infty$; see~\cite[Theorems~1.1 and~1.2]{Duran2026}. This means that the eigenvalue problem associated to~\eqref{eq:RodzinLaplacian}, namely
\begin{equation} \label{eq:Rodzin_eigen}
    \begin{cases}
        -\Delta u = \mu u & \text{in } \Omega, \\
        2 \overline \nu \partial_{\bar z} u + a u = 0 & \text{on } \partial \Omega,
    \end{cases}
\end{equation}
only admits positive eigenvalues~$\mu>0$. For all~$a>0$, we denote
\begin{equation}
    \mu_\Omega(a) := \text{the first (smallest) eigenvalue } \mu \text{ of } \eqref{eq:Rodzin_eigen}.
\end{equation}

We conclude this introductory section by discussing the~$C^2$ assumption on~$\Omega$.

\begin{remark} \label{rmk:Lipschitz}
    The techniques in this paper can be carried out if~$\Omega\subset \R^2$ is a bounded Lipschitz domain. In this case, however, it is a priori not ensured that~$\D_{\theta,m}$ and~$\RR_a$ are self-adjoint operators with discrete spectra; see~\cite{BenhellalPank2025,Pizzichillo2021} regarding Dirac operators and~\cite[Remark~1.11]{Duran2026} regarding~$\overline\partial$-Robin Laplacians. This is crucially used in this paper, on the one hand, when we implicitly assume that the eigenvalue problems~\eqref{eq:QD_eigen} and~\eqref{eq:Rodzin_eigen} have nonzero solutions and, on the other hand, when we consider their first (smallest) eigenvalues~$\lambda_\Omega(\theta,m)$ and~$\mu_\Omega(a)$, the latter given by the variational characterization in~\eqref{eq:RQ_Rodzin_mu}. These implicit assumptions are guaranteed whenever $\Omega$ is~$C^2$ regular, as shown in~\cite{Duran2026, DuranMasSanzPerela2026}.
\end{remark}

\section{Main results} \label{sec:main_results}

The main advantage of the connection studied in \cite{DuranMasSanzPerela2026} ---which relates $\lambda_\Omega(\theta,m)$ with $\mu_\Omega(a)$--- is that a simple variational characterization for~$\mu_\Omega(a)$ is available. More specifically, by~\cite[Theorem~1.2]{Duran2026} we have
\begin{equation} \label{eq:RQ_Rodzin_mu}
    \mu_\Omega(a) = \inf_{u\in E(\Omega)\setminus\{0\}} \dfrac{4\int_\Omega |\partial_{\bar z} u|^2 + a\int_{\partial\Omega} |u|^2}{\int_\Omega |u|^2},
\end{equation}
where $E(\Omega) := \{u\in L^2(\Omega,\C):\, \partial_{\bar z} u\in L^2(\Omega,\C) \text{ and }u\in L^2(\partial\Omega,\C)\}$.

Using \eqref{eq:RQ_Rodzin_mu}, one can show that the function~$a\mapsto \mu_\Omega(a)$ is continuous, strictly increasing, strictly concave, and bijective from~$(0,+\infty)$ to~$(0,\Lambda_\Omega)$, where we recall that~$\Lambda_\Omega$ is the first eigenvalue of the Dirichlet Laplacian in~$\Omega$; see~\cite[Theorem 1.3]{Duran2026}. Thanks to the connection studied in~\cite{DuranMasSanzPerela2026}, one can translate these properties and show that, for all~$m\geq 0$, the function~$\theta \mapsto \lambda_\Omega(\theta,m)$ is continuous, strictly decreasing, and bijective from~$(-\frac \pi 2,\frac {\pi}{2})$ to~$(m,\sqrt{\Lambda_\Omega+m^2})$; see~\cite[Proposition~3.8]{DuranMasSanzPerela2026}.

In this work we build on such a connection to obtain more properties of $\lambda_\Omega(\theta,m)$.

\subsection{Translation of bounds} \label{sec:translation}

The main result of this paper establishes a new way of interpreting the connection between~$\lambda_\Omega(\theta,m)$ and~$\mu_\Omega(a)$ with respect to what was done in~\cite{DuranMasSanzPerela2026}. The reader can find in~\Cref{fig:interpretation} a graphical interpretation of its statement, along which $\vartheta$ is the function defined in~\eqref{def:vartheta}, namely~$\theta  \mapsto \vartheta(\theta) = \frac{1-\sin\theta}{\cos\theta}$. 

\begin{lemma} \label{lemma:recipe_bounds}
    Let $m\geq 0$ and $\theta\in (-\frac{\pi}{2},\frac{\pi}{2})$, and let~$\Omega\subset \R^2$ be a bounded domain with~$C^2$ boundary. Then, there exists a unique~$a>0$ such that
    \begin{equation}
        \mu_\Omega(a) = \left(\dfrac{a}{\vartheta(\theta)}-m\right)^2-m^2.
    \end{equation}
    Moreover, $a>0$ is the unique solution of the previous equation if and only if
    \begin{equation}
        \lambda_\Omega(\theta,m) = \dfrac{a}{\vartheta(\theta)}-m.
    \end{equation}
\end{lemma}

\begin{figure}[ht]
	\begin{tikzpicture}
        
		\draw[->, thick] (-0.4,0) -- (14.7,0) node[below right] {$a$};
		\draw[->, thick] (0,-2.5) -- (0,7);

		\draw[line width=1, domain=0:13.3, smooth, variable=\x, CBorange] plot ({\x}, {\x/(\x/5.78+1/2}) node[right] {$\mu_\Omega(a)$};
        \node[below left] at (0,0) {$0$};

		\draw[line width=0.8, domain=0:9, smooth, variable=\x, blue] 
		plot ({\x}, {(\x/2.414-1)*(\x/2.414-1)-1}) node[right] {$p_{\theta,m}(a)= \left(\dfrac{a}{\vartheta(\theta)}-m\right)^2-m^2$};

        \draw[dashed, gray, line width=1] (7.939,0) -- (7.939,{(7.939/2.414-1)*(7.939/2.414-1)-1});
        \node[below] at (7.939,0) {$a_\star$};
        \fill (7.939,{(7.939/2.414-1)*(7.939/2.414-1)-1}) circle (0.09);

        \draw[dashed, gray, line width=1] (5.7,0) -- (5.7,{(5.7/2.414-1)*(5.7/2.414-1)-1});
        \node[below] at (5.7,0) {$a_1$};

        \draw[dashed, gray, line width=1] (10.25,0) -- (10.25, {10.25/(10.25/5.78+1/2});
        \node[below] at (10.25,0) {$a_2$};

        \node at (8.4,-1.7) {$\begin{matrix}
                \dfrac{a_1}{\vartheta(\theta)}-m < \dfrac{a_\star}{\vartheta(\theta)}-m < \dfrac{a_2}{\vartheta(\theta)}-m \\
				\rotatebox{90}{$=$} \\
				 \lambda_\Omega(\theta,m)
            \end{matrix}$};
		
	\end{tikzpicture}
	\caption{Graphical interpretation of~\Cref{lemma:recipe_bounds} for $m\geq 0$ and $\theta\in (-\frac{\pi}{2},\frac{\pi}{2})$ given.}
	\label{fig:interpretation}
\end{figure}
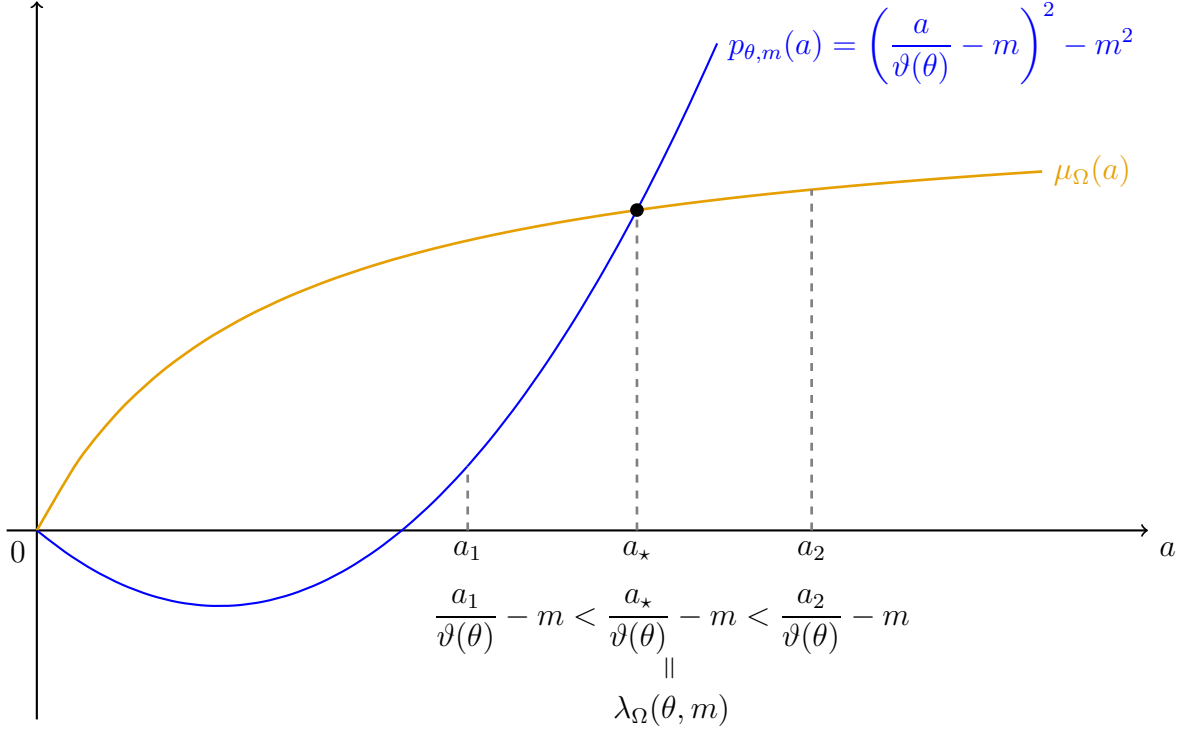

The rest of the results of this paper are consequences of~\Cref{lemma:recipe_bounds}. It is henceforth worth discussing the meaning and implications of its statement, which we do with the aid of~\Cref{fig:interpretation}. First, one can see~\Cref{lemma:recipe_bounds} as a result giving a characterization for~$\lambda_\Omega(\theta,m)$. Indeed, the statement asserts that this eigenvalue is given in terms of the unique intersection between the curves
\begin{equation}
    0<a\mapsto \mu_\Omega(a) \quad \text{and} \quad 0<a\mapsto p_{\theta,m}(a):= \left(\dfrac{a}{\vartheta(\theta)}-m\right)^2-m^2;
\end{equation}
see~$a_\star>0$ in~\Cref{fig:interpretation}. Taking into account the variational characterization~\eqref{eq:RQ_Rodzin_mu} for $\mu_\Omega(a)$, this can be seen as a variational characterization for $\lambda_\Omega(\theta,m)$; in~\Cref{rmk:Comparison_Anutnes} we compare this interpretation with the variational characterization obtained in~\cite{Antunes2021} in the case $\theta=0$ and $m=0$. Moreover,~\Cref{fig:interpretation} and the concavity and convexity of the curves suggest a way of finding bounds for~$\lambda_\Omega(\theta,m)$. Indeed, according to~\Cref{fig:interpretation}, the values $a>0$ for which~$\mu_\Omega(a) > p_{\theta,m}(a)$ seem to give lower bounds for~$\lambda_\Omega(\theta,m)$ ---see~$a_1$ in~\Cref{fig:interpretation}---, and the values~$a>0$ for which~$\mu_\Omega(a) < p_{\theta,m}(a)$ seem to give upper bounds for~$\lambda_\Omega(\theta,m)$ ---see~$a_2$ in~\Cref{fig:interpretation}. As we shall see in the proof of~\Cref{thm:translate_bounds}, this is actually what happens.

The previous comments hint that bounds for $\mu_\Omega(a)$ can be transferred to bounds for $\lambda_\Omega(\theta,m)$; and vice versa. This is formalized in the following result, which readily follows from~\Cref{lemma:recipe_bounds}. It establishes the precise bounds for~$\mu_\Omega(a)$ needed to ensure that a desired quantity~$\mathcal B>0$ is either an upper or a lower bound for $\lambda_\Omega(\theta,m)$; and vice versa.

\begin{theorem} \label{thm:translate_bounds}
    Let $\Omega\subset \R^2$ be a bounded domain with $C^2$ boundary. Given $m\geq 0$ and~$\mathcal B>0$, let~$a>0$ and~$\theta\in(-\frac \pi 2, \frac \pi 2)$ be related by
    \begin{equation}
        a = \vartheta(\theta) \big( \mathcal B + m \big) \quad \text{or, equivalently,} \quad \theta = \vartheta^{-1} \left( \dfrac{a}{\mathcal B +m} \right).
    \end{equation}
    Then, the following statements hold:
    \begin{enumerate}[label=$(\roman*)$]
        \item $\lambda_\Omega(\theta,m) = \mathcal B$ if and only if $\mu_\Omega(a) = \mathcal B^2-m^2$.
        \item $\lambda_\Omega(\theta,m) > \mathcal B$ if and only if $\mu_\Omega(a) > \mathcal B^2-m^2$.
        \item $\lambda_\Omega(\theta,m) < \mathcal B$ if and only if $\mu_\Omega(a) < \mathcal B^2-m^2$.
    \end{enumerate}
\end{theorem}

The variational characterization~\eqref{eq:RQ_Rodzin_mu} is a powerful tool for obtaining bounds for~$\mu_\Omega(a)$. We want to bring to the attention of the reader that there is no simple analog of this for the quantum dot Dirac operators, but~\Cref{thm:translate_bounds} gives a recipe for achieving it. Obtaining bounds for~$\mu_\Omega(a)$ and transferring them to bounds for~$\lambda_\Omega(\theta,m)$ according to~\Cref{thm:translate_bounds} is the idea on which~\Cref{sec:bounds} is based.

This idea might also be useful for proving Faber-Krahn or related inequalities for $\lambda_\Omega(\theta,m)$; see, for example,~\cite[Problem~5.1]{ProblemListShapeOptimization},~\cite[Conjecture~1]{Antunes2021},~\cite[questions (Q1),~(Q1'),~(Q2),~(Q2')]{Antunes2024},~\cite[Conjectures~3.17,~4.2, and~5.1]{Antunes2024} ,~\cite[Conjecture~2]{Briet2022}, and~\cite[Conjecture~1.1]{DuranMasSanzPerela2026}. Actually,~\Cref{cor:FK_all,cor:FK_some} below give some insight into this; see also~\Cref{cor:Pointwise_Equivalence,cor:Equivalence_Conjectures}.

Let us finally announce that the same idea is used in the recent paper~\cite{DuranMasSanzPerela2026PW}, with the goal of proving a sharp upper bound for $\lambda_\Omega(\theta,m)$ depending only on the isoperimetric deficit of~$\Omega$; namely, a Payne-Weinberger type inequality for $\lambda_\Omega(\theta,m)$. 

\begin{remark} \label{rmk:Comparison_Anutnes}
	The problem for $\theta=0=m$ was studied in~\cite{Antunes2021} for domains with $C^\infty$ boundary, and was one of the main inspirations of the authors in \cite{DuranMasSanzPerela2026}. We revisit~\cite[Remark~3.5]{DuranMasSanzPerela2026}, explaining the agreement of~\Cref{lemma:recipe_bounds} and~\Cref{thm:translate_bounds} with the results in~\cite{Antunes2021}. On the one hand, we recall that~\cite[Theorem~4]{Antunes2021} states that $\lambda_\Omega(0,0)$ is the (unique) $\lambda>0$ such that $\mathcal P_\Omega(\lambda)=0$, where
	\begin{equation} \label{eq:AntunesFormulation}
		\mathcal{P}_\Omega(\lambda) := \inf_{u\in E(\Omega)\setminus\{0\}}\dfrac{4\int_\Omega |\partial_{\bar z} u|^2 + \lambda \int_{\partial\Omega} |u|^2 - \lambda^2 \int_\Omega |u|^2}{\int_\Omega |u|^2}.
	\end{equation}
    This formulation coincides with~\Cref{lemma:recipe_bounds}, which asserts that $\lambda_\Omega(0,0)$ is the (unique)~$a>0$ solving~$\mu_\Omega(a) = a^2$. Indeed, noting that
    \begin{equation} \label{eq:AntunesCompare}
		\mu_\Omega(a) = \inf_{u\in E(\Omega)\setminus\{0\}}\dfrac{4\int_\Omega |\partial_{\bar z} u|^2 + a \int_{\partial\Omega} |u|^2}{\int_\Omega |u|^2} = \mathcal P_\Omega(a)+a^2,
	\end{equation}
    we have the same conclusion. On the other hand, we notice that as a combination of~\cite[Theorem~4 and Proposition~33]{Antunes2021} one can deduce that if $\lambda>0$ is such that $\mathcal P_\Omega(\lambda) < 0$ then~$\lambda_\Omega(0,0) < \lambda$; and similarly reversing the inequalities. In view of \eqref{eq:AntunesFormulation} and \eqref{eq:AntunesCompare}, this formulation is in agreement with~\Cref{thm:translate_bounds}. 
    
    We discuss now an advantage of~\Cref{thm:translate_bounds} with respect to the study in~\cite{Antunes2021} for $\theta=0$ and $m=0$. Using \eqref{eq:AntunesFormulation}, in~\cite[Section~7]{Antunes2021} it is observed that the Faber-Krahn inequality $\lambda_\Omega(0,0) \geq \lambda_\disk(0,0)$ holds true provided that $\mathcal P_\Omega(\lambda) \geq \mathcal P_\disk(\lambda)$ for all $\lambda>0$; here~$\Omega$ is a domain with area $\pi$ and~$\disk$ is the unit disk. In view of~\eqref{eq:AntunesFormulation} and \eqref{eq:AntunesCompare}, this is the same as asserting that the inequality~$\mu_\Omega(a) \geq \mu_\disk(a)$ for all $a>0$ ensures the inequality $\lambda_\Omega(\theta,0) \geq \lambda_\disk(\theta,0)$ for~$\theta = 0$. Instead, by~\Cref{thm:translate_bounds} the inequality $\mu_\Omega(a) \geq \mu_\disk(a)$ for~$a=\lambda_\disk(0,0)$ ensures the inequality $\lambda_\Omega(0,0) \geq \lambda_\disk(0,0)$; moreover, the converse is also true.
\end{remark}

\subsection{Geometric upper and lower bounds} \label{sec:bounds}

The main result in this section presents new upper and lower bounds for $\lambda_\Omega(\theta,m)$, for all $\theta\in (-\frac{\pi}{2},\frac{\pi}{2})$ and all $m\geq 0$; see~\Cref{thm:geometric_bounds}. In~\Cref{rmk:Comparison_Benguria} we compare our new lower bound with the one already available in the literature; see~\cite[Theorem~1]{Benguria2017Spectral}. The quantities on which our bounds depend are the perimeter~$|\partial\Omega|$, the first eigenvalue of the Dirichlet Laplacian $\Lambda_\Omega$, and the (inverse square of the) norm of the embedding of the harmonic Hardy space in the harmonic Bergman space
\begin{equation} \label{def:qOmega}
    q_\Omega := \inf_{h\in \mathcal H(\Omega)\setminus\{0\}} \dfrac{\int_{\partial\Omega} |h|^2}{\int_\Omega |h|^2},
\end{equation}
where
\begin{equation}
    \mathcal H(\Omega) := \text{the completion of } \big\{ v\in C^2(\overline \Omega) : \Delta v = 0 \text{ in } \Omega \big\} \text{ with respect to the norm } \|\cdot\|_{L^2(\partial\Omega)}.
\end{equation}
Before proceeding further let us comment that, by~\cite[Theorem~2]{Bucur2009}, whenever $\Omega$ is Lipschitz and satisfies the uniform outer ball condition ---in particular, whenever~$\Omega$ is~$C^2$---,~$q_\Omega$ is well-defined, positive, and coincides with the first eigenvalue $q$ of the fourth order Steklov problem
\begin{equation}
    \begin{cases}
        \Delta^2u = 0 & \text{in } \Omega, \\
        u = 0 & \text{on } \partial\Omega, \\
        \Delta u - q \partial_\nu u = 0 & \text{on } \partial\Omega.
    \end{cases}
\end{equation}

\begin{theorem} \label{thm:geometric_bounds}
    Let $\Omega\subset \R^2$ be a bounded domain with $C^2$ boundary, and let~$q_\Omega$ be as in~\eqref{def:qOmega}. The following statements hold:
    \begin{enumerate}[label=$(\roman*)$]
        \item For all $a>0$, 
            \begin{equation} \label{eq:Boundsdbar}
                \dfrac{\Lambda_\Omega}{1+\dfrac{\Lambda_\Omega}{q_\Omega} \dfrac{1}{a}} < \mu_\Omega(a) < \dfrac{\Lambda_\Omega}{1+\dfrac{4\pi}{|\partial\Omega|} \dfrac{1}{a}}.
            \end{equation}
        \item For all $\theta\in (-\frac \pi 2, \frac \pi 2)$ and all $m\geq 0$,
            \begin{equation} \label{eq:BoundsQD}
               \mathcal B^{\theta,m}_{\Omega} \left( \dfrac{\Lambda_\Omega}{2q_\Omega} \right) < \lambda_\Omega(\theta,m) < \mathcal B^{\theta,m}_{\Omega} \left(\dfrac{2\pi}{|\partial\Omega|} \right).
            \end{equation}
            Here, $(0,+\infty)\ni \xi \mapsto \mathcal B^{\theta,m}_{\Omega}(\xi)$ is the positive and strictly decreasing function defined by
            \begin{equation}
               \mathcal B^{\theta,m}_{\Omega}(\xi):= - \dfrac{\xi}{\vartheta(\theta)} + \sqrt{\left( \dfrac{\xi}{\vartheta(\theta)} + m \right)^2+\Lambda_\Omega}.
            \end{equation}
    \end{enumerate}
\end{theorem}

The same type of geometric upper and lower bounds of \Cref{thm:geometric_bounds}~$(i)$ were obtained for the first eigenvalue of the Robin Laplacian with boundary parameter~$a>0$ in~\cite[equations~(11') and~(40')]{Sperb1972}. As shall be seen in \Cref{sec:proof_bounds}, we obtain \Cref{thm:geometric_bounds}~$(i)$ adapting the arguments in~\cite{Sperb1972} to the variational characterization~\eqref{eq:RQ_Rodzin_mu} of $\mu_\Omega(a)$. Then, \Cref{thm:geometric_bounds}~$(ii)$ is obtained as a consequence of~\Cref{thm:translate_bounds}.

We remark that the bounds in  \Cref{thm:geometric_bounds}~$(i)$ are in agreement with the fact that~$\mu_\Omega(a) \in (0,\Lambda_\Omega)$ for all $a>0$. Indeed, by monotonicity we have
\begin{equation}
    0 = \underset{a \searrow 0}{\lim} \, \dfrac{\Lambda_\Omega}{1+\dfrac{\Lambda_\Omega}{q_\Omega} \dfrac{1}{a}} < \mu_\Omega(a) < \underset{a \nearrow +\infty}{\lim} \, \dfrac{\Lambda_\Omega}{1+\dfrac{4\pi}{|\partial\Omega|} \dfrac{1}{a}} = \Lambda_\Omega \quad \text{for all } a>0.
\end{equation}
Similarly, the bounds in  \Cref{thm:geometric_bounds}~$(ii)$ are in agreement with the fact that $\lambda_\Omega(\theta,m) \in (m, \sqrt{\Lambda_\Omega+m^2})$ for all $\theta \in (-\frac{\pi}{2},\frac{\pi}{2})$. Indeed, by monotonicity we have
\begin{equation}
    m = \underset{\xi \nearrow +\infty}{\lim} \, \mathcal B^{\theta,m}_{\Omega}(\xi) < \lambda_\Omega(\theta,m) < \underset{\xi \searrow0}{\lim} \, \mathcal B^{\theta,m}_{\Omega}(\xi) = \sqrt{\Lambda_\Omega+m^2} \quad \text{for all } \theta \in \textstyle{(-\frac{\pi}{2},\frac{\pi}{2})}.
\end{equation}

\begin{remark} \label{rmk:Comparison_Benguria}
     We compare our lower bound for $\lambda_\Omega(\theta,m)$ with the one previously obtained in~\cite[Theorem~1]{Benguria2017Spectral} for simply connected domains with $C^2$ boundary when~$m=0$, namely,
     \begin{equation} \label{eq:LB_Benguria}
         \lambda_\Omega(\theta,0) \geq \mathcal C^{\theta,0}_\Omega  \, \text{ for all } \theta\in \textstyle{ \left( -\frac \pi 2, \frac \pi 2 \right) }, \quad \text{where } \mathcal C^{\theta,0}_\Omega := \sqrt{\dfrac{2\pi}{|\Omega|}} \min \big\{ \vartheta(\theta), \vartheta(\theta)^{-1} \big\}.
     \end{equation}
     This bound is, in some cases, better than the one in~\Cref{thm:geometric_bounds}. Indeed, when $\theta=0$ (thus~$\vartheta(0)=1$) and $\Omega$ is the unit disk $\disk$ (hence $q_\disk=2$ and $\Lambda_\disk \approx 5.78$, see for example~\cite[Table~2]{Kuttler1968}), we have
     \begin{equation}
         1.43 \approx \lambda_\disk(0,0) > \mathcal C^{0,0}_\disk = \sqrt{2} \approx 1.41 > 1.36 \approx \mathcal B^{0,0}_\disk \left( \frac{\Lambda_\disk}{2q_\disk} \right).
     \end{equation}
     However, \Cref{fig:bounds} shows that the lower bound in~\Cref{thm:geometric_bounds} can be much better for other values of~$\theta\in(-\frac{\pi}{2},\frac{\pi}{2})$. Moreover, the geometric dependence of~$\mathcal C^{\theta,0}_\Omega$ is only on the area of~$\Omega$. This means that the lower bound in~\eqref{eq:LB_Benguria} is the same for all domains with the same area. Instead, even if we consider domains with the same area, the constant~$\mathcal B^{\theta,0}_\Omega \left( \Lambda_\Omega/2q_\Omega \right)$ can be larger (better) than~$\mathcal C^{\theta,0}_\Omega$ if the geometry of~$\Omega$ is such that the quotient~$\Lambda_\Omega/q_\Omega$ is small enough. The results in~\Cref{sec:FK} are evidence of this.
\end{remark}

\begin{figure}[h]
    \centering
    \includegraphics[width=0.75\linewidth]{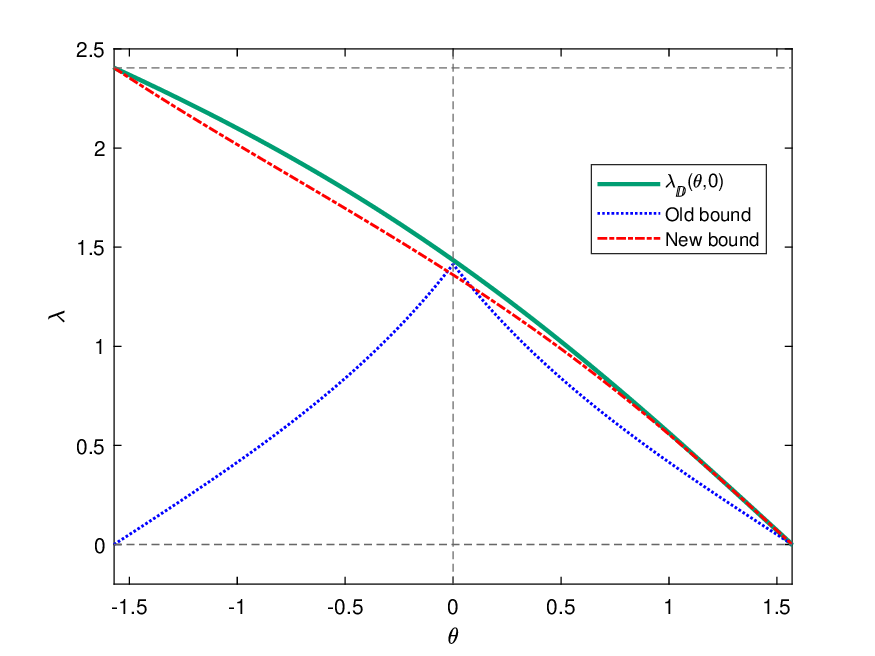}
    \caption{Comparison between the old lower bound~$\mathcal C^{\theta,0}_\disk$ in~\eqref{eq:LB_Benguria} (dotted line) and the new lower bound~$\mathcal B^{\theta,0}_\disk(\Lambda_\disk/2q_\disk)$ in~\Cref{thm:geometric_bounds}~$(ii)$ (dash-dotted line) for~$\lambda_\disk(\theta,0)$ (continuous line). The horizontal dashed lines indicate the values~$m=0$ and~$\sqrt{\Lambda_\disk}$, and the vertical dashed line indicates the infinite mass boundary condition~$\theta=0$.}
    \label{fig:bounds}
\end{figure}

\subsection{Some Faber-Krahn type inequalities} \label{sec:FK}

As an application of \Cref{thm:geometric_bounds}, in this section we obtain some Faber-Krahn type inequalities both for~$\mu_\Omega(a)$ and~$\lambda_\Omega(\theta,m)$. These are independent of the boundary parameters~$a$ and $\theta$~when the geometric quantities appearing in the lower bounds of~\Cref{thm:geometric_bounds} can be conveniently controlled.

\begin{corollary} \label{cor:FK_all}
    Let $\Omega\subset \R^2$ be a bounded domain with $C^2$ boundary, and let~$q_\Omega$ be as in~\eqref{def:qOmega}.~If
    \begin{equation} \label{eq:all_qOmega}
        \sqrt{\dfrac{|\Omega|}{\pi}} q_\Omega \geq \dfrac{\Lambda_\disk}{2}
    \end{equation}
    then the following statements hold, where $D\subset \R^2$ is a disk with the same area as $\Omega$:
    \begin{enumerate}[label=$(\roman*)$]
        \item $\mu_\Omega(a) > \mu_D(a)$ for all $a>0$.
        \item $\lambda_\Omega(\theta,m)>\lambda_D(\theta,m)$ for all $\theta\in (-\frac \pi 2, \frac \pi 2)$ and all $m\geq 0$. 
    \end{enumerate}
\end{corollary}

This result says that if $\Omega$ is a convex thin domain, then its first eigenvalue is larger than the one in a disk with the same area. See~\Cref{rmk:examples}, where we discuss some additional hypotheses on~$\Omega$ that guarantee condition~\eqref{eq:all_qOmega}; notice that disks do not satisfy~\eqref{eq:all_qOmega}.

However, since the scale invariant quantity~$\sqrt{|\Omega|}q_\Omega$ can be arbitrarily small ---see~\cite[Section~2.2]{Bucur2009}---, \Cref{cor:FK_all} does not ensure a Faber-Krahn inequality for all domains. Even in this case, we are able to obtain Faber-Krahn inequalities for some regimes of boundary parameters. This was already shown in~\cite[Theorems~2.4 and~2.5]{DuranMasSanzPerela2026} using different arguments but, as a novelty of the present paper, we can now quantify such regimes.

\begin{corollary} \label{cor:FK_some}
    Let $\Omega\subset \R^2$ be a bounded domain with $C^2$ boundary different from a disk, and let~$q_\Omega$ be as in~\eqref{def:qOmega}.~If
    \begin{equation}
        \sqrt{\dfrac{|\Omega|}{\pi}} q_\Omega < \dfrac{\Lambda_\disk}{2}
    \end{equation}
    then
    {\small \begin{equation}
        A_\Omega := \dfrac{\Lambda_\Omega}{\dfrac{|\Omega|}{\pi}\dfrac{\Lambda_\Omega}{\Lambda_\disk}-1} \left( \dfrac{1}{q_\Omega} - \dfrac{2}{\Lambda_\disk} \sqrt{\dfrac{|\Omega|}{\pi}} \right) >0, \, \, \, 
        \Theta_{\Omega,m} := \vartheta^{-1} \left( \dfrac{A_\Omega}{\sqrt{\frac{\pi}{|\Omega|} \mu_\disk \Big( \sqrt{\frac{|\Omega|}{\pi}} A_\Omega \Big) +m^2} +m} \right) \in{\textstyle(-\frac \pi 2, \frac \pi 2)},
    \end{equation} }
    and the following statements hold, where $D\subset \R^2$ is a disk with the same area as $\Omega$:
    \begin{enumerate}[label=$(\roman*)$]
        \item $\mu_\Omega(a) > \mu_D(a)$ for all $a \in [A_\Omega,+\infty)$.
        \item $\lambda_\Omega(\theta,m)>\lambda_D(\theta,m)$ for all $m\geq 0$ and all $\theta\in (-\frac \pi 2, \Theta_{\Omega,m}]$.
    \end{enumerate}
\end{corollary}

We conclude with a brief discussion on domains that satisfy condition~\eqref{eq:all_qOmega}.

\begin{remark} \label{rmk:examples}
    Condition~\eqref{eq:all_qOmega} is not void. In this remark we use the lower bounds of~$q_\Omega$ available in the literature ---see~\cite{Kuttler1968,Payne1970,Raulot2015}--- to show that convex thin enough domains satisfy~\eqref{eq:all_qOmega}. Such lower bounds of $q_\Omega$ are written in terms of the simpler geometric quantities
    \begin{equation}
        \kappa_\mathrm{min} := \min_{x\in\partial\Omega} \kappa(x), \quad \text{where } \, \kappa(x) := \text{the curvature of } x\in\partial\Omega,
    \end{equation}
    and
    \begin{equation}
        \rho := \text{the inradius of } \Omega \text{ (namely, the radius of the largest disk inscribed in } \Omega \text{)}.
    \end{equation}
    More specifically, by~\cite[Theorem~12]{Raulot2015} there holds
    \begin{equation}
        q_\Omega \geq \dfrac{1}{\rho} \left( 1-\dfrac{\rho \kappa_\mathrm{min}}{2} \right)^{-1}.
    \end{equation}
    As a consequence, if~$\Omega$ is such that
    \begin{equation} \label{eq:auxCondition}
        \sqrt{\dfrac{\pi}{|\Omega|}} \rho \left (1- \dfrac{\rho \kappa_\mathrm{min}}{2} \right) \leq \dfrac{2}{\Lambda_\disk},
    \end{equation}
    then condition~\eqref{eq:all_qOmega} is satisfied. Observe that this happens, for example, when~$\Omega$ is convex (then~$\kappa_\mathrm{min} \geq 0$) and thin enough (then~$\rho$ is small enough). Notice that~\eqref{eq:auxCondition} does not require~$\Omega$ to be simply connected, since annuli with a ratio of outer and inner radius close enough to $1$ also satisfy~\eqref{eq:auxCondition}.
\end{remark}

\section{Translation of bounds} \label{sec:proof_translation}

This section is devoted to the proofs of \Cref{lemma:recipe_bounds} and \Cref{thm:translate_bounds}. Due to the crucial role that will play in the proofs, we briefly recall the connection between quantum dot Dirac operators and~$\overline\partial$-Robin Laplacians shown in~\cite[Section~3.1]{DuranMasSanzPerela2026}. To this end, let $m\geq 0$ be given. 

On the one hand, if the vector function $\varphi=(u,v)^\intercal:\Omega\to \C^2$ solves~\eqref{eq:QD_eigen}, namely
\begin{equation}
    \begin{cases}
        -2i \partial_z v = (\lambda-m) u & \text{in } \Omega, \\
        -2i \partial_{\bar z} u = (\lambda+m)v & \text{in } \Omega, \\
        \overline \nu v = i \vartheta(\theta) u & \text{on } \partial \Omega,
    \end{cases} \quad \text{ for some } \, \lambda>m, \, \theta\in \textstyle{(-\frac{\pi}{2},\frac{\pi}{2})},
\end{equation}
then by combining the equations we see that its first component $u:\Omega\to\C$ solves
\begin{equation}
    \begin{cases}
        -\Delta u = \mu u & \text{in } \Omega, \\
        2 \overline \nu \partial_{\bar z} u + a u = 0 & \text{on } \partial \Omega,
    \end{cases} \text{ with } \mu:= \lambda^2-m^2>0, \, a := \vartheta(\theta)(\lambda+m)>0.
\end{equation}
On the other hand, if the scalar function $u:\Omega\to\C$ solves~\eqref{eq:Rodzin_eigen}, namely
\begin{equation}
    \begin{cases}
        -\Delta u = \mu u & \text{in } \Omega, \\
        2 \overline \nu \partial_{\bar z} u + a u = 0 & \text{on } \partial \Omega.
    \end{cases} \quad \text{for some } \, \mu>0, \, a>0,
\end{equation}
then by construction the vector function $\varphi=(u,v)^\intercal:= \Big( u, \frac{-2i}{\sqrt{\mu+m^2}+m} \partial_{\bar z} u \Big)^\intercal:\Omega\to \C^2$ solves
\begin{equation}
    \begin{cases}
        -2i \partial_z v = (\lambda-m) u & \text{in } \Omega, \\
        -2i \partial_{\bar z} u = (\lambda+m)v & \text{in } \Omega, \\
        \overline \nu v = i \vartheta(\theta) u & \text{on } \partial \Omega,
    \end{cases} \text{ with } \lambda:= \sqrt{\mu+m^2}>m, \, \theta:= \textstyle{ \vartheta^{-1} \Big(\frac{a}{\sqrt{\mu+m^2}+m} \Big) } \in \textstyle{(-\frac{\pi}{2},\frac{\pi}{2})}.
\end{equation}
This motivates considering the (bijective) function
\begin{equation} \label{def:T_function}
    \begin{split}
        T_m:{\textstyle(-\frac \pi 2, \frac \pi 2)}\times(m,+\infty) & \to(0,+\infty) \times (0,+\infty) \\
        (\theta,\lambda) & \mapsto \Big( \vartheta(\theta) (\lambda+m),\lambda^2-m^2 \Big),
    \end{split}
\end{equation}
whose inverse is
    \begin{equation} \label{def:T-1_function}
    \begin{split}
        T_m^{-1}: (0,+\infty)\times(0,+\infty)&\to{\textstyle(-\frac \pi 2, \frac \pi 2)}\times(m,+\infty)\\
        (a,\mu)&\mapsto  \Bigg( \vartheta^{-1}\Big({\textstyle\frac{a}{\sqrt{\mu+m^2}+m}}\Big),\sqrt{\mu + m^2} \Bigg).
    \end{split}
\end{equation}
In view of the previous argument, one might be tempted to assert that the first eigenvalues~$\lambda_\Omega(\theta,m)$ and~$\mu_\Omega(a)$ are mapped to each other through the transformation $T_m$. This is not an obvious assertion, because the images by $T_m$ of two different eigenvalues $\lambda_1(\theta,m)<\lambda_2(\theta,m)$ at the same boundary parameter~$\theta$ are two different eigenvalues $\mu_1(a_1)<\mu_2(a_2)$ at two different boundary parameters $a_1<a_2$. The mapping of the first eigenvalues is actually one of the key results in \cite{DuranMasSanzPerela2026}. More specifically, the following is proven in~\cite[Proposition~3.2]{DuranMasSanzPerela2026}; see~\Cref{fig:transformation} for a graphical representation of this result.

\begin{proposition}[\cite{DuranMasSanzPerela2026}] \label{prop:firstTOfirst}
	Let $(\theta,\lambda)\in{\textstyle(-\frac \pi 2, \frac \pi 2)}\times(m,+\infty)$ and $(a,\mu)\in(0,+\infty)\times(0,+\infty)$ be such that $T_m(\theta,\lambda) = (a,\mu)$. Then, 
	$\lambda=\lambda_\Omega(\theta,m)$ if and only if 
	$\mu=\mu_\Omega(a)$.
\end{proposition}

\begin{figure}[h]
	\begin{tikzpicture}[scale=2]
		
		\draw[->, thick] (-1.571,0) -- (1.85,0) node[anchor=north] {$\theta$};
		\draw[->, thick] (-1.57,0) -- (-1.57,1.8);
        \node[below, gray] at (-1.571,0) {$-\pi/2$};
        \draw[dashed, gray] (1.5,{sqrt(2+0.2*0.2+0.8}) -- (1.5,0) node[below] {$\pi/2$};
		
        \draw[thick, domain=-1.571:1.5, smooth, variable=\x, CBgreen]
		plot ({\x}, {-cos(\x r)/(1-sin(\x r))/2+ sqrt((cos(\x r)/(1-sin(\x r))/2+0.2)*(cos(\x r)/(1-sin(\x r))/2+0.2)+2)});
        \node[below, CBgreen] at (0.8, 0.7) {$\lambda_\Omega$};

        \draw[dashed, gray] (1.6,{sqrt(2+0.2*0.2}) -- (-1.571,{sqrt(2+0.2*0.2}) node[left] {$\sqrt{\Lambda_\Omega+m^2}$};
        \draw[dashed, gray] (1.6,0.27) -- (-1.571,0.27) node[left] {$m$};

        \node (A) at ({0.3}, {-cos(0.3 r)/(1-sin(0.3 r))/2+ sqrt((cos(0.3 r)/(1-sin(0.3 r))/2+0.2)*(cos(0.3 r)/(1-sin(0.3 r))/2+0.2)+2)+0.02}) {\phantom{$\Box$}};
        \fill ({0.3}, {-cos(0.3 r)/(1-sin(0.3 r))/2+ sqrt((cos(0.3 r)/(1-sin(0.3 r))/2+0.2)*(cos(0.3 r)/(1-sin(0.3 r))/2+0.2)+2)}) circle (0.03) node[below left] {$\big( \theta,\lambda_\Omega(\theta,m) \big)$};

		\begin{scope}[xshift=2.3cm]
			\draw[->, thick] (0,0) -- (3.4,0) node[anchor=north] {$a$};
			\draw[->, thick] (0,0) -- (0,1.8);
            \node[below, gray] at (0,0) {$0$};
            \node[left, gray] at (0,0) {$0$};

			\draw[thick, domain=0:3.3, smooth, variable=\x, CBorange] plot ({\x}, {2*\x/(1+\x)});
            \node[below, CBorange] at (0.7, 0.7) {$\mu_\Omega$};

            \draw[dashed, gray] (3.3,1.6) -- (0,1.6) node[left] {$\Lambda_\Omega$};

            \node (B) at ({1.5}, {2*1.5/(1+1.5)+0.03}) {\phantom{$\Box$}};
		    \fill ({1.5}, {2*1.5/(1+1.5)}) circle (0.03) node[below right] {$\big( a,\mu_\Omega(a) \big)$};

		\end{scope}

        \draw[->, thick, blue] (A) edge[bend left=10] node[left, sloped, fill=white] {$T_m$} (B);
        \draw[<-, thick, blue] (A) edge[bend left=-10] node[left, sloped, fill=white] {$T_m^{-1}$} (B);
        
	\end{tikzpicture}
	\caption{Graphical representation of the connection between the eigenvalues~$\lambda_\Omega(\theta,m)$ and $\mu_\Omega(a)$ given by the transformation~$T_m$.}
	\label{fig:transformation}
\end{figure}
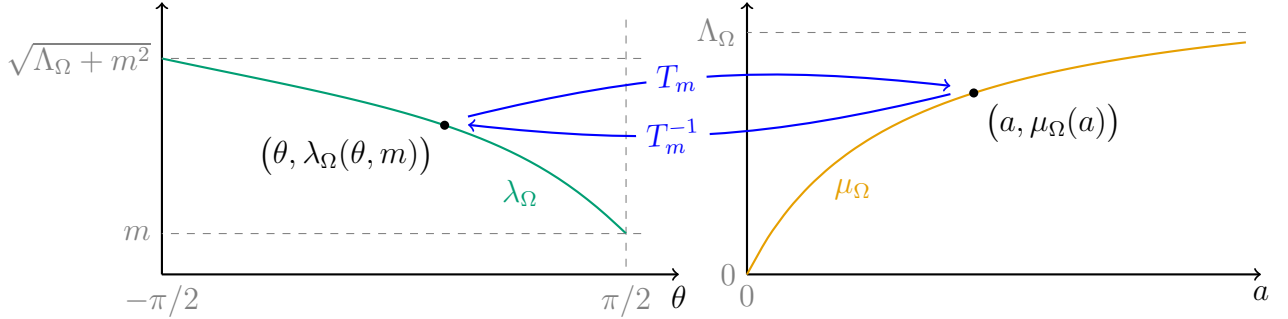

We are now ready to prove our main lemma.

\begin{proof}[Proof of \Cref{lemma:recipe_bounds}]
    Let $m\geq 0$ and $\theta\in (-\frac{\pi}{2},\frac{\pi}{2})$ be given. We first show the existence of a unique~$a_\star>0$ solving
    \begin{equation} \label{eq:unique_a}
        \mu_\Omega(a_\star) = \left(\dfrac{a_\star}{\vartheta(\theta)}-m\right)^2-m^2.
    \end{equation}
    On the one hand, by \cite[Theorem 1.3]{Duran2026}, the function $a\mapsto \mu_\Omega(a)$ is a continuous and strictly concave bijection from $(0,+\infty)$ to $(0,\Lambda_\Omega)$. Moreover, as a consequence of~\cite[Proposition~2.6~$(i)$ and~$(ii)$]{DuranMasSanzPerela2026} its slope at the origin, which we denote as
    \begin{equation}
        S_\Omega := \lim_{a\searrow0} \dfrac{\mu_\Omega(a)}{a},
    \end{equation}
    is strictly positive. On the other hand, the function
    \begin{equation}
        0 < a \mapsto p(a):= \left(\dfrac{a}{\vartheta(\theta)}-m\right)^2-m^2
    \end{equation}
    ---which was denoted $p_{\theta,m}$ in~\Cref{fig:interpretation}--- is continuous, strictly convex, unbounded from above, and its slope at the origin is 
    \begin{equation}
        p'(0) = \dfrac{-2m}{\vartheta(\theta)} \leq 0 < S_\Omega.
    \end{equation}
    As a consequence, the function $0<a \mapsto p(a) - \mu_\Omega(a)$ is continuous, strictly convex, negative for all small enough $a>0$, and positive for all large enough $a>0$. Hence, it has a unique zero in~$(0,+\infty)$. This shows the existence of a unique $a_\star>0$ solving~\eqref{eq:unique_a}. Notice that actually~$a_\star > 2m\vartheta(\theta)$, since $\mu_\Omega(a) >0$ and $p(a)\leq 0$ for all $a\in(0,2m\vartheta(\theta)]$.
    
    We now prove that this unique $a_\star$ satisfies
    \begin{equation} \label{eq:unique_a_QD}
        \lambda_\Omega(\theta,m) = \dfrac{a_\star}{\vartheta(\theta)}-m.
    \end{equation}
    To this end, we use the connection recalled in \Cref{prop:firstTOfirst}, namely,
    \begin{equation}
        \text{if } \, \tilde\lambda := \sqrt{\mu_\Omega(a_\star)+m^2} \, \text{ and } \, \tilde\theta:= \vartheta^{-1}\left(\dfrac{a_\star}{\sqrt{\mu_\Omega(a_\star)+m^2}+m} \right), \, \text{ then } \, \tilde\lambda = \lambda_\Omega(\tilde\theta,m).
    \end{equation}
    Since $a_\star> 2m\vartheta(\theta)$ solves \eqref{eq:unique_a}, on the one hand we have
    \begin{equation}
        \tilde\lambda = \sqrt{\mu_\Omega(a_\star)+m^2} = \dfrac{a_\star}{\vartheta(\theta)}-m,
    \end{equation}
    and on the other hand we have
    \begin{equation}
        \tilde\theta = \vartheta^{-1}\left(\dfrac{a_\star}{\sqrt{\mu_\Omega(a_\star)+m^2}+m} \right) = \vartheta^{-1}\left(\dfrac{a_\star}{a_\star/\vartheta(\theta)} \right) = \theta.
    \end{equation}
    Combining both equations we get that
    \begin{equation}
        \lambda_\Omega(\theta,m) = \lambda_\Omega(\tilde \theta,m) = \tilde\lambda = \dfrac{a_\star}{\vartheta(\theta)}-m.
    \end{equation}
    Finally, since $\frac{a}{\vartheta(\theta)}-m$ is linear in $a$, the solution to this last equation is unique.
\end{proof}

As an immediate consequence of \Cref{lemma:recipe_bounds}, we can translate bounds for $\mu_\Omega(a)$ to bounds for $\lambda_\Omega(\theta,m)$ ---and vice versa--- as established in \Cref{thm:translate_bounds}. 

\begin{proof}[Proof of \Cref{thm:translate_bounds}]
    Let $m\geq 0$ be given. We first show that, as an immediate consequence of~\Cref{lemma:recipe_bounds}, the following statements hold whenever $a>0$ and~$\theta\in(-\frac \pi 2, \frac \pi 2)$:
    \begin{enumerate}
        \item[$(A)$] $\mu_\Omega(a) > p(a) = \left(\dfrac{a}{\vartheta(\theta)}-m\right)^2-m^2$ if and only if $\lambda_\Omega(\theta,m) > \dfrac{a}{\vartheta(\theta)}-m$.
        \item[$(B)$] $\mu_\Omega(a) < p(a) = \left(\dfrac{a}{\vartheta(\theta)}-m\right)^2-m^2$ if and only if $\lambda_\Omega(\theta,m) <\dfrac{a}{\vartheta(\theta)}-m$.
    \end{enumerate}
    For the sake of understanding, let us relate the statements $(A)$ and $(B)$ with the graphical interpretation in~\Cref{fig:interpretation}: according to $(A)$, the values $a>0$ for which $\mu_\Omega(a) > p(a)$ give lower bounds for $\lambda_\Omega(\theta,m)$, as suggested by the point $a_1>0$ in~\Cref{fig:interpretation}; similarly, according to $(B)$, the values $a>0$ for which $\mu_\Omega(a) < p(a)$ give upper bounds for $\lambda_\Omega(\theta,m)$, as suggested by the point $a_2>0$ in~\Cref{fig:interpretation}.
    
    In the proof of~\Cref{lemma:recipe_bounds} we have seen that, given~$\theta\in(-\frac \pi 2, \frac \pi 2)$, the continuous function $a \mapsto p(a) - \mu_\Omega(a)$ is negative for small enough $a>0$ and positive for large enough $a>0$. Hence, if $a_\star$ is the unique solution in $(0,+\infty)$ to $\mu_\Omega(a_\star) = p(a_\star)$, we have
    \begin{equation} \label{eq:aux_ineq}
        p(a) - \mu_\Omega(a) >0 \, \text{ if and only if } a>a_\star, \quad \text{and} \quad p(a) - \mu_\Omega(a) <0 \, \text{ if and only if } a<a_\star.
    \end{equation}
    Since, by~\eqref{eq:unique_a_QD}, there holds
    \begin{equation}
        a_\star = \vartheta(\theta)\big( \lambda_\Omega(\theta,m) + m \big),
    \end{equation}
    the conditions~\eqref{eq:aux_ineq} rewrite as
    \begin{equation}
        \begin{split}
            p(a) - \mu_\Omega(a) >0 & \, \text{ if and only if } \dfrac{a}{\vartheta(\theta)}-m > \lambda_\Omega(\theta,m), \quad \text{and} \\ 
        p(a) - \mu_\Omega(a) <0 & \, \text{ if and only if } \dfrac{a}{\vartheta(\theta)}-m < \lambda_\Omega(\theta,m).
        \end{split}
    \end{equation}
    This proves the statements $(A)$ and $(B)$. Using them, we now address the proof of~\Cref{thm:translate_bounds}. Recall that given $m\geq 0$ and~$\mathcal B>0$, we take~$a>0$ and~$\theta\in(-\frac \pi 2, \frac \pi 2)$ related by
    \begin{equation}
        a = \vartheta(\theta) \big( \mathcal B + m \big) \quad \text{or, equivalently,} \quad \theta = \vartheta^{-1} \left( \dfrac{a}{\mathcal B +m} \right).
    \end{equation}
    For such a relation of parameters, we clearly have
    \begin{equation}
        \mathcal B = \dfrac{a}{\vartheta(\theta)} - m \quad \text{and} \quad \mathcal B^2 -m^2 = \left( \dfrac{a}{\vartheta(\theta)} - m \right)^2 - m^2.
    \end{equation}
    Then item~$(i)$ readily follows from \Cref{lemma:recipe_bounds}, and items~$(ii)$ and~$(iii)$ readily follow from statements~$(A)$ and~$(B)$.
\end{proof}

As we shall see in the proof of \Cref{cor:FK_all,cor:FK_some}, with the aim of proving Faber-Krahn type inequalities between two bounded~$C^2$ domains $\Omega\subset \R^2$ and $\Omega_0 \subset \R^2$ it might be useful to consider \Cref{thm:translate_bounds} with the choices~$\mathcal B = \lambda_{\Omega_0}(\theta,m)$ or $\mathcal B = \sqrt{\mu_{\Omega_0}(a)+m^2}$. This gives rise to the following result.

\begin{corollary} \label{cor:Pointwise_Equivalence}
    Let $\Omega, \Omega_0 \subset \R^2$ be two bounded domains with $C^2$ boundary. Given $m\geq 0$, consider $\lambda_\Omega(\theta):= \lambda_\Omega(\theta,m)$ as a function only of $\theta$ (and similarly for $\Omega_0$). The following hold:
    \begin{enumerate}[label=$(\roman*)$]
        \item Given $\theta\in(-\frac \pi 2, \frac \pi 2)$, then 
        \begin{equation} \label{def:a_as_theta_conj}
            a:= \vartheta(\theta) \big( \lambda_{\Omega_0}(\theta) + m \big) = \mu_{\Omega_0}^{-1} \big( \lambda_{\Omega_0}(\theta)^2-m^2 \big)>0
        \end{equation}
        and
        \begin{enumerate}
            \item[$(a)$] $\lambda_\Omega(\theta) = \lambda_{\Omega_0}(\theta)$ if and only if $\mu_\Omega(a) = \mu_{\Omega_0}(a)$,
            \item[$(b)$] $\lambda_\Omega(\theta) > \lambda_{\Omega_0}(\theta)$ if and only if $\mu_\Omega(a) > \mu_{\Omega_0}(a)$,
            \item[$(c)$] $\lambda_\Omega(\theta) < \lambda_{\Omega_0}(\theta)$ if and only if $\mu_\Omega(a) < \mu_{\Omega_0}(a)$.
        \end{enumerate} \vspace{0.25cm}
        \item Given $a>0$, then 
        \begin{equation} \label{def:theta_as_a_conj}
            \theta:= \vartheta^{-1} \left( \dfrac{a}{\sqrt{\mu_{\Omega_0}(a) + m^2}+m} \right) = \lambda_{\Omega_0}^{-1} \big({\textstyle\sqrt{\mu_{\Omega_0}(a)+m^2}}\big) \in{\textstyle(-\frac \pi 2, \frac \pi 2)}
        \end{equation}
        and
        \begin{enumerate}
            \item[$(a')$] $\mu_\Omega(a) = \mu_{\Omega_0}(a)$ if and only if $\lambda_\Omega(\theta) = \lambda_{\Omega_0}(\theta)$,
            \item[$(b')$] $\mu_\Omega(a)> \mu_{\Omega_0}(a)$ if and only if $\lambda_\Omega(\theta)> \lambda_{\Omega_0}(\theta)$,
            \item[$(c')$] $\mu_\Omega(a) < \mu_{\Omega_0}(a)$ if and only if $\lambda_\Omega(\theta) < \lambda_{\Omega_0}(\theta)$.
        \end{enumerate}
    \end{enumerate}
\end{corollary}

Before proving \Cref{cor:Pointwise_Equivalence}, let us compare it with \cite[Theorem~2.1]{DuranMasSanzPerela2026}.

\begin{remark}
    We note that \Cref{cor:Pointwise_Equivalence}~$(i)$~$(b)$ and~$(ii)$~$(b')$ is an improvement of~\cite[Theorem~2.1]{DuranMasSanzPerela2026}. First, while here~$\Omega_0$ can be any domain, in~\cite[Theorem~2.1]{DuranMasSanzPerela2026}~$\Omega_0\equiv D$ is taken to be a disk with the same~area as~$\Omega$. Furthermore, in~\cite[Theorem~2.1]{DuranMasSanzPerela2026} the boundary parameter~$a$ of item~$(i)$ is taken to~be
    \begin{equation} \label{aJFA}
        \tilde a:=\mu_\Omega^{-1}(\lambda_D(\theta)^2-m^2)>0,
    \end{equation}
    and the boundary parameter $\theta$ of item~$(ii)$ is taken to be 
    \begin{equation} \label{thetaJFA}
        \tilde \theta:=\lambda_\Omega^{-1} \big({\textstyle\sqrt{\mu_D(a)+m^2}}\big) \in{\textstyle(-\frac \pi 2, \frac \pi 2)}.
    \end{equation}
    Notice that these parameters depend on \emph{both} domains~$\Omega$ and~$D\equiv \Omega_0$, unlike~\eqref{def:a_as_theta_conj} and~\eqref{def:theta_as_a_conj}. With~\Cref{cor:Pointwise_Equivalence} we can take parameters uniformly in $\Omega$ in order to test inequalities for a fixed $\Omega_0$. Nevertheless, although the choice of parameters~\eqref{def:a_as_theta_conj} versus~\eqref{aJFA} is different, the corresponding results are equivalent. Indeed, let us show that
    \begin{equation}
        \mu_\Omega(\tilde a) > \mu_D(\tilde a) \text{ with } \tilde a \text{ as in \eqref{aJFA}} \quad \text{if and only if} \quad \mu_\Omega(a) > \mu_D(a) \text{ with } a \text{ as in \eqref{def:a_as_theta_conj}}.
    \end{equation}
    This follows from the fact that the function~$a\mapsto \mu_\Omega(a)$ is strictly increasing for every domain~$\Omega$; recall~\cite[Theorem~1.3]{Duran2026}. Using this, for the ``only if" implication we observe that
    \begin{equation}
        a = \mu_D^{-1}(\lambda_D(\theta)^2-m^2) = \mu_D^{-1}(\mu_\Omega(\tilde a)) > \mu_D^{-1}(\mu_D(\tilde a)) = \tilde a,
    \end{equation}
    and then we conclude that
    \begin{equation}
        \mu_\Omega(a) > \mu_\Omega(\tilde a) = \lambda_D^2(\theta)-m^2 = \mu_D(a).
    \end{equation}
    For the ``if" implication we observe that
    \begin{equation}
        \tilde a = \mu_\Omega^{-1} (\lambda_D(\theta)^2-m^2) = \mu_\Omega^{-1} (\mu_D(a)) < \mu_\Omega^{-1} (\mu_\Omega(a)) = a,
    \end{equation}
    and then we conclude that
    \begin{equation}
        \mu_D(\tilde a) < \mu_D(a) = \lambda_D(\theta)^2-m^2 = \mu_\Omega(\tilde a).
    \end{equation}
    An analogous argument shows that the choice of parameters~\eqref{def:theta_as_a_conj} versus~\eqref{thetaJFA} is equivalent, namely, that
    \begin{equation}
        \lambda_\Omega(\tilde \theta) > \lambda_D(\tilde \theta) \text{ with } \tilde \theta \text{ as in \eqref{thetaJFA}} \quad \text{if and only if} \quad \lambda_\Omega(\theta) > \lambda_D(\theta) \text{ with } \theta \text{ as in \eqref{def:theta_as_a_conj}}.
    \end{equation}
    In conclusion,~\cite[Theorem~2.1]{DuranMasSanzPerela2026} and~\Cref{cor:Pointwise_Equivalence}~$(i)$~$(b)$ and~$(ii)$~$(b')$ are equivalent when $\Omega_0$ is a disk with the same area as~$\Omega$.
\end{remark}

We address now the proof of \Cref{cor:Pointwise_Equivalence}. Recall that, given $m\geq 0$, we consider~$\theta\mapsto \lambda_\Omega(\theta):= \lambda_\Omega(\theta,m)$ as a function only of $\theta\in(-\frac \pi 2, \frac \pi 2)$.

\begin{proof}[Proof of \Cref{cor:Pointwise_Equivalence}]
    Given $\theta\in(-\frac \pi 2, \frac \pi 2)$, once~\eqref{def:a_as_theta_conj} is shown, item~$(i)$ readily follows choosing $\mathcal B = \lambda_{\Omega_0}(\theta)>0$ in~\Cref{thm:translate_bounds} with
    \begin{equation}
        a := \vartheta(\theta) \big( \lambda_{\Omega_0}(\theta) + m \big).
    \end{equation}
    We henceforth only have to prove the second equality in~\eqref{def:a_as_theta_conj}, namely, that
    \begin{equation}
        a = \mu_{\Omega_0}^{-1} \big( \lambda_{\Omega_0}(\theta)^2-m^2 \big);
    \end{equation}
    notice that the right-hand side is a well-defined positive number, because $0 < \lambda_{\Omega_0}(\theta)^2-m^2 < \Lambda_{\Omega_0}$ and~$a\mapsto \mu_{\Omega_0}(a)$ is a bijection from $(0,+\infty)$ to $(0,\Lambda_{\Omega_0})$; recall the paragraph below~\eqref{eq:RQ_Rodzin_mu}. Setting $a_\star := \mu_{\Omega_0}^{-1} \big( \lambda_{\Omega_0}(\theta)^2-m^2 \big)$, we observe that
    \begin{equation}
        \mu_{\Omega_0}(a_\star) = \lambda_{\Omega_0}(\theta)^2-m^2 = \mu_{\Omega_0}(a),
    \end{equation}
    the last equality holding by~\Cref{thm:translate_bounds}~$(i)$. Since~$a\mapsto \mu_{\Omega_0}(a)$ is a bijection, we conclude that~$a_\star=a$, as desired.

    In a similar way, given $a>0$, once~\eqref{def:theta_as_a_conj} is shown item~$(ii)$ readily follows choosing $\mathcal B = \sqrt{\mu_{\Omega_0}(a)+m^2}>0$ in~\Cref{thm:translate_bounds} with
    \begin{equation}
        \theta:= \vartheta^{-1} \left( \dfrac{a}{\sqrt{\mu_{\Omega_0}(a) + m^2}+m} \right).
    \end{equation}
    We henceforth only have to prove the second equality in~\eqref{def:theta_as_a_conj}, namely, that
    \begin{equation}
        \theta = \lambda_{\Omega_0}^{-1} \big({\textstyle\sqrt{\mu_{\Omega_0}(a)+m^2}}\big);
    \end{equation}
    notice that the right-hand side is a well-defined number in $(-\frac{\pi}{2}, \frac{\pi}{2})$, since $m< \sqrt{\mu_{\Omega_0}(a)+m^2} < \sqrt{\Lambda_{\Omega_0}+m^2}$ and~$\theta\mapsto\lambda_{\Omega_0}(\theta)$ is a bijection from $(-\frac{\pi}{2}, \frac{\pi}{2})$ to $(m,\sqrt{\Lambda_{\Omega_0}+m^2})$. Setting $\theta_\star := \lambda_{\Omega_0}^{-1} \big({\textstyle\sqrt{\mu_{\Omega_0}(a)+m^2}}\big)$, we observe that 
    \begin{equation}
        \lambda_{\Omega_0}(\theta_\star) = \sqrt{\mu_{\Omega_0}(a)+m^2} = \lambda_{\Omega_0}(\theta),
    \end{equation}
    the last equality holding by~\Cref{thm:translate_bounds}~$(i)$. Since $\theta\mapsto\lambda_{\Omega_0}(\theta)$ is a bijection, we conclude that~$\theta_\star = \theta$, as desired.
\end{proof}

As an immediate consequence of~\Cref{cor:Pointwise_Equivalence}, we obtain the following equivalences of Faber-Krahn inequalities. This result was already shown in~\cite[Corollary~2.2]{DuranMasSanzPerela2026} in the case when~$\Omega$ is not a disk and~$\Omega_0$ is a disk with the same area as $\Omega$; see also~\cite[Remark~2.3]{DuranMasSanzPerela2026}.

\begin{corollary} \label{cor:Equivalence_Conjectures}
    Let $\Omega, \Omega_0 \subset \R^2$ be two bounded domains with $C^2$ boundary. The following statements are equivalent:
    \begin{enumerate}[label=$(\roman*)$]
        \item $\mu_\Omega(a) > \mu_{\Omega_0}(a)$ for all $a>0$.
        \item $\lambda_\Omega(\theta,m)>\lambda_{\Omega_0}(\theta,m)$ for all $\theta\in (-\frac \pi 2, \frac \pi 2)$, for some $m\geq 0$.
        \item $\lambda_\Omega(\theta,m)>\lambda_{\Omega_0}(\theta,m)$ for all $\theta\in (-\frac \pi 2, \frac \pi 2)$ and all $m\geq 0$.
    \end{enumerate}
\end{corollary}

\begin{proof}
    Given $m_0\geq 0$, assume that $\lambda_\Omega(\theta,m_0)>\lambda_{\Omega_0}(\theta,m_0)$ for all $\theta\in (-\frac \pi 2, \frac \pi 2)$. Applying~\Cref{cor:Pointwise_Equivalence}~$(ii)$ with $m=m_0$, we deduce that $\mu_\Omega(a) > \mu_{\Omega_0}(a)$ for all $a>0$, thus $(ii)$ implies $(i)$. Applying now~\Cref{cor:Pointwise_Equivalence}~$(i)$ with an arbitrary $m\geq0$, this in turn leads to $\lambda_\Omega(\theta,m)>\lambda_{\Omega_0}(\theta,m)$ for all $\theta\in (-\frac \pi 2, \frac \pi 2)$ and all $m\geq 0$, which gives~$(iii)$. Obviously, $(iii)$ implies $(ii)$.
\end{proof}

\section{Geometric upper and lower bounds} \label{sec:proof_bounds}

This section is devoted to the proof of~\Cref{thm:geometric_bounds}. The bounds for $\lambda_\Omega(\theta,m)$ in~\Cref{thm:geometric_bounds}~$(ii)$ will follow from the bounds for $\mu_\Omega(a)$ in~\Cref{thm:geometric_bounds}~$(i)$ together with the translation given in~\Cref{thm:translate_bounds}. The bounds for $\mu_\Omega(a)$ in~\Cref{thm:geometric_bounds}~$(i)$ will be proven following the ideas in \cite{Sperb1972}, which lead to geometric upper and lower bounds for the first eigenvalue~$\mu_\Omega^\mathrm{Rob}(a)$ of the Robin Laplacian with boundary parameter $a>0$. We recall that the variational characterizations of $\mu_\Omega^\mathrm{Rob}(a)$ and $\mu_\Omega(a)$ are, respectively,
\begin{equation}
    \mu_\Omega^\mathrm{Rob}(a) = \inf_{u\in H^1(\Omega,\R)\setminus\{0\}} \dfrac{\int_\Omega |\nabla u|^2 + a\int_{\partial\Omega} |u|^2}{\int_\Omega |u|^2},
\end{equation}
and
\begin{equation} \label{proof:mu}
    \mu_\Omega(a) = \inf_{u\in E(\Omega)\setminus\{0\}} \dfrac{4\int_\Omega |\partial_{\bar z} u|^2 + a\int_{\partial\Omega} |u|^2}{\int_\Omega |u|^2},
\end{equation}
where $E(\Omega) := \{u\in L^2(\Omega,\C):\, \partial_{\bar z} u\in L^2(\Omega,\C)\text{ and }u\in L^2(\partial\Omega,\C)\}.$

\begin{proof}[Proof of \Cref{thm:geometric_bounds}]
    We first address the proof of $(i)$, following the ideas in \cite{Sperb1972}. 
    
    \underline{Upper bound in $(i)$:} We use that $\mu_\Omega(a) \leq \mu_\Omega^\mathrm{Rob}(a)$. This inequality immediately follows, on the one hand, from the inclusion $H^1(\Omega,\R) \subsetneq E(\Omega)$ and, on the other hand, from the fact that~$4|\partial_{\bar z} u|^2 = |\nabla u|^2$ almost everywhere, for all~$u\in H^1(\Omega,\R)$. Then, the upper bound for~$\mu_\Omega(a)$ trivially follows from the upper bound for~$\mu_\Omega^\mathrm{Rob}(a)$ proven in~\cite[equation~(11')]{Sperb1972}, which is precisely
    \begin{equation} \label{eq:boundRobin}
        \mu_\Omega^\mathrm{Rob}(a) < \dfrac{\Lambda_\Omega}{1+\dfrac{4\pi}{|\partial\Omega|} \dfrac{1}{a}}.
    \end{equation}
    For the sake of self-containedness, and since the reference~\cite{Sperb1972} is only available in German, we sketch the proof of~\eqref{eq:boundRobin}.
    \begin{enumerate}
        \item First, let $u_D$ be the first eigenfunction of the Dirichlet Laplacian (which is unique up to a multiplicative constant), and for $\beta\in \R$ use the trial function $u := u_D + \frac{\beta}{|\Omega|} \int_\Omega u_D$ in the variational characterization of $\mu_\Omega^\mathrm{Rob}(a)$. This yields a lower bound for $1/\mu_\Omega^\mathrm{Rob}(a)$.
        \item Second, use the Payne-Rayner inequality $(\int_\Omega u_D)^2 \geq \frac{4\pi}{\Lambda_\Omega} \int_\Omega |u_D|^2$ ---which is only valid for planar domains, see \cite{Payne1972}---, to estimate the lower bound obtained in step (1) from below.
        \item The resulting lower bound for $1/\mu_\Omega^\mathrm{Rob}(a)$ is a function of the parameter $\beta\in\R$. Choose the optimal $\beta$ that maximizes the lower bound obtained in step (2).
        \item Finally, use the inequality $\Lambda_\Omega \geq \frac{4\pi}{|\Omega|}$ ---which follows, for example, combining Cauchy-Schwarz with the Payne-Rayner inequality in step (2)--- to drop positive terms in the maximal lower bound obtained in step (3). The result yields~\eqref{eq:boundRobin}.
    \end{enumerate}

    \underline{Lower bound in $(i)$:} Let $u\in E(\Omega)\setminus\{0\}$ be a minimizer of the variational characterization~\eqref{proof:mu} of $\mu_\Omega(a)$, which exists by~\cite[Theorem 1.2]{Duran2026}. We decompose $u$ as a sum of a harmonic function and a zero trace function, namely
    \begin{equation} \label{eq:decomposition}
        u = h+w, \quad \text{where} \quad 
        \begin{cases}
            -\Delta h = 0 & \text{in } \Omega, \\
            h = u & \text{on } \partial\Omega,
        \end{cases} \quad \text{and} \quad 
        \begin{cases}
            -\Delta w = -\Delta u & \text{in } \Omega, \\
            w = 0 & \text{on } \partial\Omega.
        \end{cases}
    \end{equation}
    By existence and uniqueness of the solutions of the Dirichlet and Poisson problems on Lipschitz domains ---in particular, when $\Omega$ is $C^2$---, such a decomposition exists and is unique. Furthermore, it is not trivial. Indeed, by~\cite[Theorem~1.2]{Duran2026} the minimizer $u$ solves
    \begin{equation}
        \begin{cases}
            -\Delta u = \mu_\Omega(a) u & \text{in } \Omega, \\
            2\overline \nu \partial_{\bar z} u + au = 0 & \text{on } \partial\Omega,
        \end{cases} \quad \text{with} \quad 0<\mu_\Omega(a) < \Lambda_\Omega,
    \end{equation}
    hence $\Delta u \not\equiv 0$ in $\Omega$ and $u\not\equiv 0$ on $\partial\Omega$; indeed, if the minimizer vanishes at the boundary then it is an eigenfunction of the Dirichlet Laplacian, contradicting the fact that $\mu_\Omega(a) < \Lambda_\Omega$. Therefore, $h\not\equiv 0$ and $w\not\equiv 0$, and moreover $h$ and $w$ are not proportional in $L^2(\Omega,\C)$. Notice, however, that~$h$ and~$w$ are not necessarily orthogonal in $L^2(\Omega,\C)$, but they satisfy
    \begin{equation} \label{eq:orthogonality}
        4 \int_{\Omega} \partial_{\bar z} w \, \overline{\partial_{\bar z} h} = - \int_\Omega w \, \overline{\Delta h} + 2 \int_{\partial\Omega} \nu w \, \overline{ \partial_{\bar z} h} = 0 \quad \text{and} \quad \int_{\partial\Omega} w\, \overline h = 0;
    \end{equation}
    in the first equality we have used the divergence theorem, writing $\Delta = 4\partial_{\bar z}\partial_z$. Using that~$u$ is a minimizer of the variational characterization of~$\mu_\Omega(a)$, the decomposition~\eqref{eq:decomposition}, the orthogonality conditions in~\eqref{eq:orthogonality}, and the strict triangle inequality ---since $h$ and $w$ are not proportional in~$L^2(\Omega,\C)$---, we can bound the reciprocal of the Rayleigh quotient of the variational formulation for~$\mu_\Omega(a)$ evaluated at~$u$ as follows:
    \begin{equation}
        \dfrac{1}{\mu_\Omega(a)} = \dfrac{\int_\Omega |u|^2}{4\int_\Omega |\partial_{\bar z} u|^2 + a \int_{\partial\Omega} |u|^2} < \dfrac{\left( \left(\int_\Omega |h|^2\right)^{1/2} + \left(\int_\Omega |w|^2\right)^{1/2} \right)^2}{4\int_\Omega |\partial_{\bar z} h|^2 + a \int_{\partial\Omega} |h|^2 + 4\int_\Omega |\partial_{\bar z} w|^2}.
    \end{equation}
    We now split the quotient in the right-hand side as a sum of quotients of $h$ and $w$ solely, in virtue of the elementary inequality
    \begin{equation}
        \dfrac{(\alpha+\beta)^2}{\gamma^2+\delta^2} \leq \dfrac{\alpha^2}{\gamma^2} + \dfrac{\beta^2}{\delta^2} \quad \text{for all } \alpha,\beta \in \R, \text{ and } \gamma,\delta\in \R\setminus\{0\},
    \end{equation}
    which follows from applying the Cauchy-Schwarz inequality to the scalar product~$(\frac{\alpha}{\gamma},\frac{\beta}{\delta})\cdot(\gamma,\delta)$. Thus,
    \begin{equation} \label{proof:aux_estimate}
        \dfrac{1}{\mu_\Omega(a)} < \dfrac{\int_\Omega |h|^2}{4\int_\Omega |\partial_{\bar z} h|^2 + a \int_{\partial\Omega} |h|^2} + \dfrac{\int_\Omega |w|^2}{4\int_\Omega |\partial_{\bar z} w|^2} \leq \dfrac{\int_\Omega |h|^2}{a \int_{\partial\Omega} |h|^2} + \dfrac{\int_\Omega |w|^2}{\int_\Omega |\nabla w|^2},
    \end{equation}
    where in the last inequality we have used, on the one hand, that $4\int_\Omega |\partial_{\bar z} h|^2 \geq 0$ and, on the other hand, that 
    \begin{equation}
        4\int_\Omega |\partial_{\bar z} w|^2 = \int_\Omega |\nabla w|^2,
    \end{equation}
    which holds true by the divergence theorem because $w$ has zero trace; see~\cite[Lemma~3.5]{Duran2026}. Using the variational characterizations 
    \begin{equation}
        q_\Omega = \inf_{\Delta h = 0 \text{ in } \Omega} \dfrac{\int_{\partial\Omega} |h|^2}{\int_\Omega |h|^2} \quad \text{and} \quad \Lambda_\Omega = \inf_{w = 0 \text{ on } \partial\Omega} \dfrac{\int_\Omega |\nabla w|^2}{\int_\Omega |w|^2},
    \end{equation}
    from \eqref{proof:aux_estimate} we conclude that
    \begin{equation} \label{aux:conclusion}
        \dfrac{1}{\mu_\Omega(a)} < \dfrac{1}{aq_\Omega} + \dfrac{1}{\Lambda_\Omega}.
    \end{equation}
    Rearranging terms, we obtain the lower bound in~$(i)$. We now address the proof of~$(ii)$.

    \underline{Bounds in $(ii)$:} The idea is to use \Cref{thm:translate_bounds} to translate the previously obtained bounds of $\mu_\Omega(a)$ to bounds of $\lambda_\Omega(\theta,m)$. To this end, we want to rewrite the bounds of $\mu_\Omega(a)$, which are written in the form
    \begin{equation} \label{aux:mathcalA}
        \mathcal A_\Omega \left(\frac{\Lambda_\Omega}{2q_\Omega}, a \right) < \mu_\Omega(a) < \mathcal A_\Omega \left( \frac{2\pi}{|\partial\Omega|}, a \right) \quad \text{with} \quad \mathcal A_\Omega (\xi,a) := \dfrac{\Lambda_\Omega}{1+\dfrac{2\xi}{a}},
    \end{equation}
    in such a way that, for a given $\theta \in (-\frac{\pi}{2}, \frac{\pi}{2})$,
    \begin{equation}
        \mathcal A_\Omega \big( \xi,a \big) = \mathcal B^2-m^2 \quad \text{with} \quad a = \vartheta(\theta) \big( \mathcal B + m \big).
    \end{equation}
    We can achieve this by taking $a=a(\xi)>0$ such that
    \begin{equation} \label{proof:unique_axi}
        \mathcal A_\Omega \big( \xi,a(\xi) \big) = \left( \dfrac{a(\xi)}{\vartheta(\theta)} -m \right)^2 -m^2,
    \end{equation}
    and in this way, by \Cref{thm:translate_bounds} we have that
    \begin{equation} \label{proof:bounds_axi}
        \dfrac{a\left(\frac{\Lambda_\Omega}{2q_\Omega}\right)}{\vartheta(\theta)} - m < \lambda_\Omega(\theta,m) < \dfrac{a\left(\frac{2\pi}{|\partial\Omega|}\right)}{\vartheta(\theta)} - m.
    \end{equation}
    Note that $a(\xi)>0$ solves \eqref{proof:unique_axi} if and only if
    \begin{equation}
        \dfrac{a(\xi) \Lambda_\Omega}{a(\xi)+ 2\xi} = \left( \dfrac{a(\xi)}{\vartheta(\theta)} - 2m \right)\dfrac{a(\xi)}{\vartheta(\theta)},
    \end{equation}
    and a straightforward computation shows that the unique positive solution to this equation is
    \begin{equation}
        a(\xi) = m\vartheta(\theta) -\xi + \sqrt{\big( m\vartheta(\theta) + \xi \big)^2+\Lambda_\Omega \vartheta(\theta)^2}.
    \end{equation}
    Substituting this in \eqref{proof:bounds_axi} we obtain the bounds in item~$(ii)$, and the proof is completed.
\end{proof}

\section{Some Faber-Krahn type inequalities} \label{sec:proof_FK}

This section is devoted to the proofs of the Faber-Krahn type inequalities appearing in~\Cref{cor:FK_all,cor:FK_some}. Given a domain $\Omega$ different from a disk and a disk $D$ of the same area, both proofs are based on the idea of showing that the upper bound of $\mu_D(a)$ in~\Cref{thm:geometric_bounds}~$(i)$ is smaller than the lower bound of~$\mu_\Omega(a)$ in~\Cref{thm:geometric_bounds}~$(i)$. Then, the Faber-Krahn type inequalities for the first eigenvalue of the $\overline\partial$-Robin Laplacians follow by the chain of inequalities
\begin{equation} \label{eq:chain}
    \mu_D(a) < \dfrac{\Lambda_D}{1+\dfrac{4\pi}{|\partial D|} \dfrac{1}{a}} \leq \dfrac{\Lambda_\Omega}{1+\dfrac{\Lambda_\Omega}{q_\Omega} \dfrac{1}{a}} < \mu_\Omega(a),
\end{equation}
and the Faber-Krahn type inequalities for the first positive eigenvalue of the quantum dot Dirac operators follow either by~\Cref{cor:Pointwise_Equivalence} or by~\Cref{cor:Equivalence_Conjectures}. The hypothesis on the geometry of~$\Omega$ in~\Cref{cor:FK_all}, namely
\begin{equation} \label{eq:RECall_qOmega}
    \sqrt{\dfrac{|\Omega|}{\pi}} q_\Omega \geq \dfrac{\Lambda_\disk}{2},
\end{equation}
ensures the validity of the second inequality in~\eqref{eq:chain} for all $a>0$. When, instead, the geometry of~$\Omega$ does not satisfy~\eqref{eq:RECall_qOmega}, the validity of the second inequality in~\eqref{eq:chain} is ensured only in some range of parameters.

At this point, the reader may have noticed that, for the sake of obtaining Faber-Krahn type inequalities using this strategy, it would be sufficient to require that the lower bound of~$\mu_\Omega(a)$ in~\Cref{thm:geometric_bounds}~$(i)$ is greater than $\mu_D(a)$, without the need of using any upper bound of $\mu_D(a)$. However, since $\mu_D(a)$ is defined through an implicit equation in $a$ (see~\cite[Proposition~A.2]{Duran2026}), the threshold condition analogous to~\eqref{eq:RECall_qOmega} that one would get is not explicit.

\begin{proof}[Proof of \Cref{cor:FK_all}]
    Let $D$ be a disk of the same area as $\Omega$, and assume that \eqref{eq:RECall_qOmega} holds true (notice that then $\Omega$ cannot be a disk). As explained at the beginning of the section, to prove~$(i)$ it suffices to show that the second inequality in~\eqref{eq:chain} holds. That is,
    \begin{equation} \label{eq:goal_aux1}
        \dfrac{\Lambda_D}{1+\dfrac{4\pi}{|\partial D|} \dfrac{1}{a}} \leq \dfrac{\Lambda_\Omega}{1+\dfrac{\Lambda_\Omega}{q_\Omega} \dfrac{1}{a}} \quad \text{for all } a>0.
    \end{equation}    
    Using the scaling property~$|\Omega|\Lambda_D = |D|\Lambda_D = \pi\Lambda_\disk$ and~$|\partial D| = \sqrt{4\pi |D|} = \sqrt{4\pi |\Omega|}$, from~\eqref{eq:RECall_qOmega} we have that
    \begin{equation} \label{eq:arguing}
         0 \geq \dfrac{1}{q_\Omega} - \dfrac{2}{\Lambda_\disk} \sqrt{\dfrac{|\Omega|}{\pi}} = \dfrac{1}{q_\Omega} - \dfrac{2}{\Lambda_D} \sqrt{\dfrac{\pi}{|\Omega|}} = \dfrac{1}{q_\Omega} - \dfrac{1}{\Lambda_D} \dfrac{4\pi}{|\partial D|}.
    \end{equation}
    As a consequence, for all $a>0$ there holds
    \begin{equation}
        \left( \dfrac{\Lambda_\Omega}{q_\Omega} - \dfrac{\Lambda_\Omega}{\Lambda_D} \dfrac{4\pi}{|\partial D|} \right) \dfrac{1}{a} \leq 0 < \dfrac{\Lambda_\Omega}{\Lambda_D} -1,
    \end{equation}
    where the last inequality holds true by the classical Faber-Krahn inequality~\cite{Faber1923,Krahn1925}. Rearranging terms we obtain \eqref{eq:goal_aux1}, concluding the proof of item~$(i)$. Once this is proved, item~$(ii)$ follows by~\Cref{cor:Equivalence_Conjectures}.
\end{proof}

\begin{proof}[Proof of \Cref{cor:FK_some}]
    Let $\Omega$ be different from a disk, $D$ a disk of the same area, and assume that \eqref{eq:RECall_qOmega} is not satisfied. As explained at the beginning of the section, to prove~$(i)$ it suffices to show that the second inequality in~\eqref{eq:chain} holds. That is,
    \begin{equation} \label{eq:goal_aux2}
        \dfrac{\Lambda_D}{1+\dfrac{4\pi}{|\partial D|} \dfrac{1}{a}} \leq \dfrac{\Lambda_\Omega}{1+\dfrac{\Lambda_\Omega}{q_\Omega} \dfrac{1}{a}} \quad \text{for all } a\geq A_\Omega,
    \end{equation}
    where
    \begin{equation}
        A_\Omega := \dfrac{\Lambda_\Omega}{\dfrac{|\Omega|}{\pi}\dfrac{\Lambda_\Omega}{\Lambda_\disk}-1} \left( \dfrac{1}{q_\Omega} - \dfrac{2}{\Lambda_\disk} \sqrt{\dfrac{|\Omega|}{\pi}} \right)
    \end{equation}
    is clearly well-defined and positive, because so is the denominator by the classical Faber-Krahn inequality~\cite{Faber1923,Krahn1925}. Arguing as in \eqref{eq:arguing}, we now have
    \begin{equation}
        0 < \dfrac{1}{q_\Omega} - \dfrac{2}{\Lambda_\disk} \sqrt{\dfrac{|\Omega|}{\pi}} = \dfrac{1}{q_\Omega} - \dfrac{1}{\Lambda_D} \dfrac{4\pi}{|\partial D|},
    \end{equation}
    hence
    \begin{equation}
        \left( \dfrac{\Lambda_\Omega}{q_\Omega} - \dfrac{\Lambda_\Omega}{\Lambda_D} \dfrac{4\pi}{|\partial D|} \right) \dfrac{1}{a} \leq \dfrac{|\Omega|}{\pi}\dfrac{\Lambda_\Omega}{\Lambda_\disk}-1 = \dfrac{\Lambda_\Omega}{\Lambda_D} -1 \quad \text{for all } a \geq A_\Omega;
    \end{equation}
    in the last equality we have used the scaling property $|\Omega|\Lambda_D = |D|\Lambda_D = \pi\Lambda_\disk$. Rearranging terms we obtain~\eqref{eq:goal_aux2}, concluding the proof of item~$(i)$. 
    
    We conclude with the proof of item~$(ii)$. To this end, we translate the above result according to~\Cref{cor:Pointwise_Equivalence}~$(ii)$. That is, we have that~$\lambda_\Omega(\theta,m) > \lambda_D(\theta,m)$
    \begin{equation} \label{eq:regime}
        \text{for all } \theta = \vartheta^{-1} \left( \dfrac{a}{\sqrt{\mu_D(a) + m^2}+m} \right) \text{ with } a\geq A_\Omega.
    \end{equation}
    Since, as shown in~\cite[Lemma~3.7]{DuranMasSanzPerela2026}, the function
    \begin{equation}
        a \mapsto \vartheta^{-1} \left( \dfrac{a}{\sqrt{\mu_D(a) + m^2}+m} \right)
    \end{equation}
    is continuous, strictly decreasing, and bijective from $(0,+\infty)$ to $(-\frac \pi 2, \frac \pi 2)$, the parameters $\theta\in(-\frac{\pi}{2}, \frac{\pi}{2})$ satisfying~\eqref{eq:regime} are 
    \begin{equation}
        -\frac{\pi}{2} < \theta \leq \vartheta^{-1} \left( \dfrac{A_\Omega}{\sqrt{\mu_D(A_\Omega)+m^2}+m} \right)= \vartheta^{-1} \left( \dfrac{A_\Omega}{\sqrt{\frac{\pi}{|\Omega|} \mu_\disk \Big( \sqrt{\frac{|\Omega|}{\pi}} A_\Omega \Big) +m^2} +m} \right) =: \Theta_{\Omega,m},
    \end{equation}
    where we have used the scaling property $\mu_{t\Omega}(a) = \mu_\Omega(ta)/t^2$ for all $t>0$. This concludes the proof of item~$(ii)$.
\end{proof}

\section*{Acknowledgments}

The author thanks Albert Mas and Tom\'as Sanz-Perela for enlightening discussions on the topic of this paper.

\end{document}